\documentclass{amsart}
\usepackage{amssymb,latexsym}
\usepackage{enumerate}

\numberwithin{equation}{section}

\newtheorem{theorem}{Theorem}[section]
\newtheorem{proposition}[theorem]{Proposition}
\newtheorem{corollary}[theorem]{Corollary}
\newtheorem{lemma}[theorem]{Lemma}
\theoremstyle{definition}
\newtheorem{definition}[theorem]{Definition}
\newtheorem{remark}[theorem]{Remark}
\newtheorem{example}[theorem]{Example}

\newcommand{\inter}{{\rm int}}

\newcommand{\N}{{\mathbb N}}
\newcommand{\R}{{\mathbb R}}

\begin{document}

\title[The Dirichlet problem on the Sierpinski gasket ]
{On the existence of three solutions for the Dirichlet problem on the Sierpinski gasket}
\author[B. E. Breckner]{Brigitte E. Breckner}
\address{Babe\c{s}-Bolyai University, Faculty of Mathematics and Computer Science\\
Kog\u{a}lniceanu str. 1\\
400084 Cluj-Napoca, Romania} \email{breckner@gmx.net}

\author[D. Repov\v{s}]{Du\v{s}an Repov\v{s}}
\address{Institute of Mathematics, Physics and Mechanics\\
P.O.Box 2964, SI-1001 Ljubljana}
\email{dusan.repovs@guest.arnes.si}

\author[Cs. Varga]{Csaba Varga}
\address{Babe\c{s}-Bolyai University, Faculty of Mathematics and Computer Science\\
Kog\u{a}lniceanu str. 1\\
400084 Cluj-Napoca, Romania} \email{varga\_gy\_csaba@yahoo.com}

\date{}

\begin{abstract}
We apply a recently obtained  three critical points theorem of B. Ricceri to prove the existence of at least three solutions of certain two-parameters Dirichlet
problems defined on the Sierpinski gasket. We also show the existence of at least three nonzero solutions of certain perturbed two-parameters Dirichlet problems
on the Sierpinski gasket, using both the mountain pass theorem of Ambrosetti-Rabinowitz and that of Pucci-Serrin.
\end{abstract}

\subjclass[2010]{Primary 35J20; Secondary 28A80, 35J25, 35J60,
47J30, 49J52}

\keywords{Sierpinski gasket, weak Laplacian, Dirichlet problem
on the Sierpinski gasket, weak solution, critical point, minimax theorems, mountain pass theorems}

\maketitle

\section{Introduction}
The celebrated three critical points theorem obtained by Ricceri in \cite{Ri00} turned out to be one of the most often applied abstract multiplicity results
for the study of different types of nonlinear problems of variational nature. In this sense we refer to the references listed in \cite{Ri09r}. Also, this three
critical points theorem has been extended to certain classes of non-smooth functions (see, for example, \cite{ArCa}, \cite{BoCa}, \cite{MaMo}). Ricceri has 
published both a revised form of his three critical points theorem (\cite{Ri09r}) and a refinement of it (\cite{Ri09f}). A corollary of the latter, stated
also in \cite{Ri09f}, is the following result:

\begin{theorem}\label{Ricceri3cp}
Let $X$ be a separable and reflexive real Banach space, and $\Phi, J\colon X\to\R$ functionals satisfying the following conditions:
\begin{itemize}
\item[{\rm (i)}] $\Phi$ is a coercive, sequentially weakly lower semicontinuous
$C^1$-functional, bounded on each bounded subset of $X$, and whose derivative admits a continuous inverse on $X^*$.
\item[{\rm (ii)}] If $(u_n)$ is a sequence in $X$ converging weakly to $u$, and if $\displaystyle{\liminf_{n\to\infty}\Phi(u_n)\leq \Phi(u)}$, then $(u_n)$ has a
subsequence converging strongly to $u$.
\item[{\rm (iii)}] $J$ is a $C^1$-functional with compact derivative.
\item[{\rm (iv)}] The functional $\Phi$ has a strict local minimum $u_0$ with $\Phi(u_0)=J(u_0)=0$.
\item[{\rm (v)}] The inequality $\rho_1<\rho_2$ holds, where
$$\rho_1:=\max\left\{0,\, \limsup_{||u||\to\infty} \frac{J(u)}{\Phi(u)},\, \limsup_{u\to u_0}\frac{J(u)}{\Phi(u)}\right\}\hbox { and }
\rho_2:=\sup_{u\in\Phi^{-1}(]0,\infty[)}\frac{J(u)}{\Phi(u)}.$$
\end{itemize}
Then, for each compact interval $[\lambda_1,\lambda_2]\subset]\frac1{\rho_2},\frac1{\rho_1}[$ {\rm (}where, by convention, $\frac10:=\infty$ and
$\frac1\infty=0${\rm)},
there exists a positive real number $r$ with the following property: For every $\lambda\in [\lambda_1,\lambda_2]$ and for every $C^1$-functional
$\Psi\colon X\to\R$ with compact derivative there exists $\delta>0$ such that, for every $\eta\in[0,\delta]$, the equation
$$\Phi'(u)=\lambda J'(u)+\eta\Psi'(u)$$
has at least three solutions in $X$ whose norms are less than $r$.
\end{theorem}


In the present paper we show with the aid of Theorem \ref{Ricceri3cp} that, under suitable assumptions on the functions
$f,g\colon V\times\R\to\R$, the following two-parameters Dirichlet problem defined on the Sierpinski
gasket $V$ in $\R^{N-1}$ has at least three solutions
$$(DP_{\lambda,\eta})
\left\{
\begin{array}{l}
-\Delta u(x)=\lambda f(x,u(x))+\eta g(x,u(x)),\ \forall x\in V\setminus V_0,\\
\\
u|_{V_0}=0.
\end{array}
\right.
$$
So far we know, this would be the first application of a Ricceri type three critical points theorem to nonlinear partial differential equations on fractals.
(Among the contributions to the theory of nonlinear elliptic equations on fractals we mention \cite{Breradvar}, \cite{Fa99}, \cite{FaHu}, \cite{Hu},
\cite{HuaZh}, \cite{stripaper}).

We also study, in a particular case, a perturbed version of problem $(DP_{\lambda,\eta})$. A similar problem, but involving the $p$-Laplacian,
 has been recently investigated in \cite{An}.
\medskip

\noindent {\bf Notations.} We denote by $\N$ the set of natural
numbers $\{0,\, 1,\, 2,\dots\}$, by $\N^*:=\N\setminus\{0\}$ the set
of positive naturals, and by $|\cdot|$ the Euclidian norm on the
spaces $\R^n$, $n\in\N^*$.

If $X$ is a topological space and $M$ a subset of it, then $\overline M$ and $\partial M$ denote
the closure, respectively, the boundary of $M$.

If $X$ is a normed space and $r$ a positive real, then $B_r$ stands for the open ball
with radius $r$ centered at the origin.

\section{The Sierpinski gasket}

In its initial representation that goes back to the pioneering
papers of the Polish mathematician Waclaw Sierpinski (1882--1969),
the {\it Sierpinski gasket} is the connected subset of the plane
obtained from an equilateral triangle by removing the open middle
inscribed equilateral triangle of $4^{-1}$ the area, removing the
corresponding open triangle from each of the three constituent
triangles, and continuing this way. The gasket can also be obtained
as the closure of the set of vertices arising in this construction.
Over the years, the Sierpinski gasket showed both to be
extraordinarily useful in representing roughness in nature and man's
works. We refer to \cite{stri} for an elementary
introduction to this subject and to  \cite{stribook} for
important applications to differential equations on fractals.

We now rigorously describe the construction of the Sierpinski gasket
in a general setting. Let $N\geq2$ be a natural number and let
$p_1,\dots, p_N\in\R^{N-1}$ be so that $|p_i-p_j|=1$ for $i\neq j$.
Define, for every $i\in\{1,\dots,N\}$, the map
$S_i\colon\R^{N-1}\to\R^{N-1}$ by
$$S_i(x)=\frac12\,x+\frac12\,p_i\,.$$
Obviously every $S_i$ is a similarity with ratio $\frac12$. Let
${\mathcal S}:=\{S_1,\dots, S_N\}$ and denote by $F\colon{\mathcal
P}(\R^{N-1})\to{\mathcal P}(\R^{N-1})$ the map assigning to a subset
$A$ of $\R^{N-1}$ the set
$$F(A)=\bigcup_{i=1}^NS_i(A).$$
It is known (see, for example, Theorem 9.1 in \cite{Fa}) that there
is a unique nonempty compact subset $V$ of $\R^{N-1}$, called the
{\it attractor of the family} ${\mathcal S}$, such that $F(V)=V$
(that is, $V$ is a fixed point of the map $F$). The set $V$ is
called the {\it Sierpinski gasket} (SG for short) in $\R^{N-1}$. It
can be constructed inductively as follows: Put
$V_0:=\{p_1,\dots,p_N\}$, $V_m:=F(V_{m-1})$, for $m\geq1$, and
$V_*:=\cup_{m\geq0}V_m$. Since $p_i=S_i(p_i)$ for
$i=\overline{1,N}$, we have $V_0\subseteq V_1$, hence $F(V_*)=V_*$.
Taking into account that the maps $S_i$, $i=\overline{1,N}$, are
homeomorphisms, we conclude that $\overline{V_*}$ is a fixed point
of $F$. On the other hand, denoting by $C$ the convex hull of the
set $\{p_1,\dots, p_N\}$, we observe that $S_i(C)\subseteq C$ for
$i=\overline{1,N}$. Thus $V_m\subseteq C$ for every $m\in\N$, so
$\overline{V_*}\subseteq C$. It follows that $\overline{V_*}$ is
nonempty and compact, hence $V=\overline{V_*}$. In the sequel $V$
is considered to be endowed with the relative topology induced from
the Euclidean topology on $\R^{N-1}$. The set $V_0$ is called the
{\it intrinsic boundary} of the SG.

The family ${\mathcal S}$ of similarities satisfies the open set
condition (see pg. 129 in \cite{Fa}) with the interior $\inter\, C$
of $C$. (Note that $\inter\, C\neq\emptyset$ since the points
$p_1,\dots, p_N$ are affine independent.) Thus, by Theorem 9.3 of
\cite{Fa}, the Hausdorff dimension $d$ of $V$ satisfies the equality
$$\sum_{i=1}^N\left(\frac12\right)^d=1,$$
hence $d=\frac{\ln N}{\ln 2}$, and $0<{\mathcal H}^d(V)<\infty$,
where ${\mathcal H}^d$ is the $d$-dimensional Hausdorff measure on
$\R^{N-1}$. Let $\mu$ be the normalized restriction of ${\mathcal
H}^d$ to the subsets of $V$, so $\mu(V)=1$. The following property
of $\mu$ will be important for our investigations
\begin{equation}\label{support}
\mu(B)>0,\hbox{ for every nonempty open subset $B$ of }V.
\end{equation}
In other words, the support of $\mu$ coincides with $V$.
We refer, for example, to \cite{Breradvar} for the proof of
(\ref{support}).

\section{The space $H_0^1(V)$} We retain the notations from the previous section and briefly recall from \cite{FaHu} the following notions
(see also \cite{Hu} and \cite{Ko} for the case $N=3$). Denote by
$C(V)$ the space of real-valued continuous functions on $V$ and by
$$C_0(V):=\{u\in C(V)\mid u|_{V_0}=0\}.$$
The spaces $C(V)$ and $C_0(V)$ are endowed with the usual supremum
norm $||\cdot||_{sup}$. For a function $u\colon V\to\R$ and for $m\in\N$
let
\begin{equation}\label{defWm}
W_m(u)=\left(\frac{N+2}{N}\right)^m\sum_{\underset{|x-y|=2^{-m}}{x,y\in
V_m}}(u(x)-u(y))^2.
\end{equation}
We have $W_m(u)\leq W_{m+1}(u)$ for every natural $m$, so we can put
\begin{equation}\label{defW}
W(u)=\lim_{m\to\infty} W_m(u).
\end{equation}
Define now
$$H_0^1(V):=\{u\in C_0(V)\mid W(u)<\infty\}.$$
It turns out that $H_0^1(V)$ is a dense linear subset of
$L^2(V,\mu)$ (equipped with the usual $||\cdot||_2$ norm). We now
endow $H_0^1(V)$ with the norm
$$||u||=\sqrt{W(u)}.$$
In fact, there is an inner product defining this norm: For $u,v\in
H_0^1(V)$ and $m\in\N$ let
$${\mathcal W}_m(u,v)=\left(\frac{N+2}{N}\right)^m\sum_{\underset{|x-y|=2^{-m}}{x,y\in V_m}}(u(x)-u(y))(v(x)-v(y)).$$
Put
$${\mathcal W}(u,v)=\lim_{m\to\infty} {\mathcal W}_m(u,v).$$
Then ${\mathcal W}(u,v)\in\R$, and $H_0^1(V)$, equipped with the
inner product ${\mathcal W}$ (which obviously induces the norm
$||\cdot||$), becomes a real Hilbert space. Moreover, if $c:=2N+3$, then
\begin{equation}\label{embeddingconstant}
||u||_{sup}\leq c||u||, \hbox{ for every }u\in H_0^1(V),
\end{equation}
and the embedding
\begin{equation}\label{embedding}
(H_0^1(V),||\cdot||)\hookrightarrow (C_0(V),||\cdot||_{sup})
\end{equation}
is compact.

We now state a useful property of the space $H_0^1(V)$ which shows,
together with the facts that $(H_0^1(V),||\cdot||)$ is a Hilbert
space and that $H_0^1(V)$ is dense in $L^2(V,\mu)$, that ${\mathcal
W}$ is a Dirichlet form on $L^2(V,\mu)$. (See, for example, Lemma 3.1
of \cite{Breradvar} for the straightforward proof.)

\begin{lemma}\label{MaMianalogue}
Let $h\colon \R\to\R$ be a Lipschitz mapping with Lipschitz constant
$L\geq0$ and such that $h(0)=0$. Then, for every $u\in H_0^1(V)$, we
have $h\circ u\in H_0^1(V)$ and $||h\circ u||\leq L\cdot||u||$.
\end{lemma}

\section{The Dirichlet problem on the Sierpinski gasket} Keep the notations from the previous sections.
We also recall from \cite{FaHu} (respectively, from \cite{Hu} and
\cite{Ko} in the case $N=3$) that one can define in a standard way a bijective,
linear, and self-adjoint operator $\Delta\colon D\to L^2(V,\mu)$, where
$D$ is a linear subset of $H_0^1(V)$ which is dense in $L^2(V,\mu)$
(and dense also in $(H_0^1(V),||\cdot||)$), such that
$$-{\mathcal W}(u,v)=\int_V\Delta u\cdot vd\mu,\hbox{ for every }(u,v)\in D\times H_0^1(V).$$
The operator $\Delta$ is called the {\it weak Laplacian on} $V$.

\begin{remark}\label{separable}
Theorem 19.B of \cite{Zevol2}, applied to $\Delta^{-1}\colon L^2(V,\mu)\to L^2(V,\mu)$, yields in particular that $H_0^1(V)$ is separable
(see also sections 19.9 and 19.10 in \cite{Zevol2}).
\end{remark}

Given a continuous function $h\colon V\times\R\to\R$, we
can formulate now the following {\it Dirichlet problem on the SG}:
Find  appropriate functions $u\in H_0^1(V)$ (in fact, $u\in D$) such that
$$(P)
\left\{
\begin{array}{l}
-\Delta u(x)=h(x,u(x)),\ \forall x\in V\setminus V_0,\\
\\
u|_{V_0}=0.
\end{array}
\right.
$$
A function $u\in H_0^1(V)$ is called a {\it weak solution of} $(P)$
if
$${\mathcal W}(u,v)-\int_Vh(x,u(x))v(x)d\mu=0,\ \forall v\in H_0^1(V).$$

\begin{remark}
Using the regularity result Lemma 2.12 of \cite{FaHu}, it follows that every weak solution of problem $(P)$ is actually a strong solution
(as defined in \cite{FaHu}). For this reason we will call in the sequel weak solutions of problem (P) simply {\it solutions of problem} $(P)$.
\end{remark}

Before defining the energy functional attached to problem $(P)$ we
recall a few basic notions.

\begin{definition}
Let $E$ be a real Banach space and  $T\colon E\to\R$ a functional.

(1) We say that $T$ is {\it Fr\'{e}chet differentiable at} $u\in E$ if
there exists a continuous linear map $T'(u)\colon E\to\R$, called
the {\it Fr\'{e}chet differential of $T$ at $u$},  such that
$$\lim_{v\to 0}\frac{|T(u+v)-T(u)-T'(u)(v)|}{||v||}=0.$$
The functional $T$ is {\it Fr\'{e}chet differentiable {\rm(}on $E${\rm)}} if $T$
is Fr\'{e}chet differentiable at every point $u\in E$. In this case the mapping $T'\colon E\to E^*$
assigning to each point $u\in E$ the Fr\'{e}chet differential of $T$ at $u$
is called the {\it Fr\'{e}chet derivative}, or, shortly, the {\it derivative}
 of $T$ on $E$. If $T'\colon E\to E^*$ is continuous, then $T$ is called
a $C^1$-{\it functional}.

(2) If $T$ is Fr\'{e}chet differentiable on $E$, then
a point $u\in
E$ is a {\it critical point of} $T$ if $T'(u)=0$. The value of $I$ at $u$ is then called a {\it critical value of} $I$.
\end{definition}

\begin{remark}\label{useful}
Note that if the Fr\'{e}chet differentiable functional $T\colon E\to\R$ has in $u\in E$ a local
extremum, then $u$
is a critical point of $T$.
\end{remark}

\begin{proposition}\label{compderivative}
Let $h\colon V\times \R\to\R$ be continuous and define $H\colon V\times\R \to\R$ by
$$H(x,t)=\int_0^th(x,\xi)d\xi.$$
Then the mapping $J\colon H_0^1(V)\to\R$ given by
$$J(u)=\int_VH(x,u(x))d\mu$$
satisfies the following properties:
\begin{itemize}
\item[{\rm a)}] $J$ is a $C^1$-functional.
\item[{\rm b)}] Its derivative $J'\colon H_0^1(V)\to (H_0^1(V))^*$ is compact.
\item[{\rm c)}] $J$ is sequentially weakly continuous.
\end{itemize}
\end{proposition}
\begin{proof}
a) The proof of Proposition 2.19 in \cite{FaHu} implies that $J$ is a $C^1$-functional and that its derivative $J'\colon H_0^1(V)\to (H_0^1(V))^*$
is given by
$$J'(u)(v)=\int_Vh(x,u(x))v(x)d\mu,\hbox{ for all }u,v\in H_0^1(V).$$

b) To show that $J'$ is compact, pick a bounded sequence $(u_n)$ in $H_0^1(V)$. Since $H_0^1(V)$ is reflexive and since the embedding
(\ref{embedding}) is compact, there exists a subsequence of $(u_n)$ which converges in $(C_0(V),||\cdot||_{sup})$. Without any loss of
generality we can assume that $(u_n)$ converges in $(C_0(V),||\cdot||_{sup})$ to an element $u\in C_0(V)$. Define $T\colon H_0^1(V)\to\R$ by
$$T(v)=\int_Vh(x,u(x))v(x)d\mu,\hbox{ for all }v\in H_0^1(V).$$
According to (\ref{embeddingconstant}), the functional $T$ belongs to $(H_0^1(V))^*$. We next show that the sequence $(J'(u_n))$ converges to $T$ in $(H_0^1(V))^*$.
By (\ref{embeddingconstant}) the following inequality holds for every index $n$
$$||J'(u_n)-T||\leq c\int_V|h(x,u_n(x))-h(x,u(x))|d\mu.$$
Using the Lebesgue dominated convergence theorem, we conclude that $(J'(u_n))$ converges to $T$ in $(H_0^1(V))^*$. Thus $J'$ is compact.

c) The assertion follows from b) and Corollary 41.9 of \cite{Ze}. We also give a direct proof: Clearly $H$ is continuous. Let $(u_n)$ be a sequence which converges weakly to $u$ in $H_0^1(V)$. Since the embedding
(\ref{embedding}) is compact, $(u_n)$ converges to $u$ in $(C_0(V),||\cdot||_{sup})$. The Lebesgue dominated convergence theorem implies now
that $(J(u_n))$ converges to $J(u)$. Thus $J$ is sequentially weakly continuous.
\end{proof}

\begin{proposition}\label{criticalpoint}
Let $h\colon V\times\R\to\R$ be continuous. Then the functional $I\colon H_0^1(V)\to\R$ given by
$$I(u)=\frac12||u||^2-J(u),$$
where $J\colon H_0^1(V)\to\R$ is defined in Proposition {\rm \ref{compderivative}},
is a $C^1$-functional and its derivative $I'\colon H_0^1(V)\to (H_0^1(V))^*$
is given by
$$I'(u)(v)={\mathcal W}(u,v)-\int_Vh(x,u(x))v(x)d\mu,\hbox{ for all }u,v\in H_0^1(V).$$
In particular, $u\in H_0^1(V)$ is a solution of problem $(P)$
if and only if $u$ is a critical point of $I$.
\end{proposition}
\begin{proof} See Proposition 2.19 in \cite{FaHu}.
\end{proof}

\begin{remark}\label{panda}
The functional $I\colon H_0^1(V)\to\R$ defined in Proposition \ref{criticalpoint} is called the {\it energy functional} attached to problem $(P)$.
\end{remark}

We now state for later use some fundamental properties of the energy functional $I$.

\begin{corollary}\label{Iwslsc}
Let $h\colon V\times\R\to\R$ be continuous. Then the functional $I\colon H_0^1(V)\to\R$ defined in Proposition {\rm \ref{criticalpoint}} is sequentially
weakly lower semicontinuous.
\end{corollary}
\begin{proof}
The function $u\in H_0^1(V)\mapsto||u||^2\in\R$ is continuous in the norm
topology on $H_0^1(V)$ and convex, thus it is sequentially weakly
lower semicontinuous on $H_0^1(V)$. The conclusion follows now from assertion c) of Proposition \ref{compderivative}.
\end{proof}

\begin{corollary}\label{forPS}
Let $h\colon V\times\R\to\R$ be continuous and consider the functional $I\colon H_0^1(V)\to\R$ defined in Proposition {\rm \ref{criticalpoint}}. If $(u_n)$
is a bounded sequence in $H_0^1(V)$ such that the sequence $(I'(u_n))$ converges to $0$, then $(u_n)$ contains a convergent subsequence.
\end{corollary}
\begin{proof}
Using Proposition \ref{criticalpoint}, we know that for every index $n$
$$I'(u_n)={\mathcal W}(u_n,\cdot)-J'(u_n).$$
Assertion b) of Proposition \ref{compderivative} yields now the conclusion.
\end{proof}

\section{A Dirichlet problem depending on two parameters}
Let $f,g\colon V\times\R\to\R$ be continuous, and  define the functions $F,G\colon V\times\R \to\R$ by
$$F(x,t)=\int_0^tf(x,\xi)d\xi\quad \hbox{and}\quad G(x,t)=\int_0^tg(x,\xi)d\xi.$$
For every $\lambda,\eta\geq0$ consider the following Dirichlet problem on the SG
$$(DP_{\lambda,\eta})
\left\{
\begin{array}{l}
-\Delta u(x)=\lambda f(x,u(x))+\eta g(x,u(x)),\ \forall x\in V\setminus V_0,\\
\\
u|_{V_0}=0.
\end{array}
\right.
$$
By Proposition \ref{criticalpoint}, the energy functional attached to the problem $(DP_{\lambda,\eta})$ is the map $I\colon H_0^1(V)\to\R$ defined by
$$I(u)=\frac12||u||^2-\lambda\int_VF(x,u(x))d\mu-\eta\int_VG(x,u(x))d\mu.$$
The aim of this section is to apply Theorem \ref{Ricceri3cp} to show that, under suitable assumptions and for certain values of the parameters $\lambda$ and $\eta$,
problem $(DP_{\lambda,\eta})$ has at least three weak solutions. More precisely, we can state the following result.

\begin{theorem}\label{main1}
Assume that the following hypotheses hold:
\begin{itemize}
\item[{\rm (C1)}] The function $f\colon V\times \R\to\R$ is continuous.
\item[{\rm (C2)}] The function $F\colon V\times\R \to\R$ satisfies the following conditions:
\begin{itemize}
\item[{\rm (1)}] There exist $\alpha\in[0,2[$, $a\in L^1(V,\mu)$, and $m\geq0$ such that
$$F(x,t)\leq m(a(x)+|t|^\alpha),\hbox{ for all }(x,t)\in V\times\R.$$
\item[{\rm (2)}] There exist $t_0>0$, $M\geq0$ and $\beta>2$ such that
$$F(x,t)\leq M|t|^\beta, \hbox{ for all }(x,t)\in V\times[-t_0,t_0].$$
\item[{\rm (3)}] There exists $t_1\in\R\setminus\{0\}$ such that for all $x\in V$ and for all $t$ between $0$ and $t_1$ we have
$$F(x,t_1)>0\hbox{ and } F(x,t)\geq0.$$
\end{itemize}
\end{itemize}
Then there exists a real number $\Lambda\geq 0$ such that, for each compact interval $[\lambda_1,\lambda_2]\subset ]\Lambda,\infty[$, there exists a positive
real number $r$ with the following property: For every $\lambda\in [\lambda_1,\lambda_2]$ and every continuous function $g\colon V\times\R\to\R$ there exists
$\delta>0$ such that, for each $\eta\in[0,\delta]$, the problem $(DP_{\lambda,\eta})$ has at least three solutions whose norms are less than $r$.
\end{theorem}
\begin{proof}
Set $X:=H_0^1(V)$. Then $X$ is separable (by Remark \ref{separable}) and reflexive (as a Hilbert space). Define the functions $\Phi,J\colon X\to\R$ for every $u\in X$ by
$$\Phi(u)=\frac12||u||^2,\quad J(u)=\int_VF(x,u(x))d\mu.$$
In order to apply Theorem \ref{Ricceri3cp}, we show that the conditions (i)--(v) required in this theorem are satisfied for the above defined functions.

Clearly condition (i) of Theorem \ref{Ricceri3cp} is satisfied. (Note that $\Phi'\colon X\to X^*$ is defined by $\Phi'(u)(v)={\mathcal W}(u,v)$ for every
$u,v\in X$.) Condition (ii) is a consequence of the facts that $X$ is uniformly convex and that $\Phi$ is sequentially weakly lower semicontinuous.
Condition (iii) follows from assertions a) and b) of Proposition \ref{compderivative}. Obviously condition (iv) holds for $u_0=0$.

To verify (v), observe first that assumption (1) of $(C2)$ implies,
together with (\ref{embeddingconstant}), that for every $u\in X\setminus\{0\}$ the
following inequality holds:
$$
\frac{J(u)}{\Phi(u)}\leq\frac{2m}{||u||^2}\int_Vad\mu+2mc^\alpha||u||^{\alpha-2}.
$$
Since $\alpha<2$, we conclude that
\begin{equation}\label{limpsup1}
\limsup_{||u||\to\infty}\frac{J(u)}{\Phi(u)}\leq0.
\end{equation}

Note that if $u\in X$ is so that $||u||\leq \frac{t_0}c$, then, by (\ref{embeddingconstant}),  $||u||_{sup}\leq t_0$.
It follows that $u(x)\in[-t_0,t_0]$ for every $x\in V$.
Using (2) of $(C2)$, we thus get that for every $x\in V$
$$F(x,u(x))\leq M|u(x)|^\beta\leq Mc^\beta||u||^\beta.$$
Hence the following inequality holds for every $u\in X\setminus\{0\}$ with $||u||\leq \frac{t_0}c$
$$\frac{J(u)}{\Phi(u)}\leq  2M c^\beta||u||^{\beta-2}.$$
Since $\beta>2$, we obtain
\begin{equation}\label{limpsup2}
\limsup_{u\to0}\frac{J(u)}{\Phi(u)}\leq0.
\end{equation}
The inequalities (\ref{limpsup1}) and (\ref{limpsup2}) yield that
\begin{equation}\label{rho1}
\rho_1:=\max\left\{0,\, \limsup_{||u||\to\infty} \frac{J(u)}{\Phi(u)},\, \limsup_{u\to u_0}\frac{J(u)}{\Phi(u)}\right\}=0.
\end{equation}

Without  loss of generality we may assume that the real number $t_1$ in condition (3) of $(C2)$ is positive.
Lemma \ref{MaMianalogue} implies that $|u|\in H_0^1(V)$ whenever $u\in H_0^1(V)$.
Thus we can pick a function $u\in H_0^1(V)$ such that $u(x)\geq0$ for every $x\in V$, and
such that there is an element $x_0\in
V$ with $u(x_0)>t_1$. It follows that $U:=\{x\in V\mid u(x)>t_1\}$ is a
nonempty open subset of $V$. Let $h\colon\R\to\R$ be defined by
$h(t)=\min\{t,t_1\}$, for every $t\in\R$. Then $h(0)=0$ and $h$ is a
Lipschitz map with Lipschitz constant $L=1$. Lemma
\ref{MaMianalogue} yields that $u_1:=h\circ u\in H_0^1(V)$. Moreover,
$u_1(x)=t_1$ for every $x\in U$, and $0\leq u_1(x)\leq t_1$ for every $x\in
V$. Then, according to condition (3) of $(C2)$, we obtain
$$F(x,u_1(x))>0,\hbox{ for every }x\in U,\quad\hbox{and}\quad F(x,u_1(x))\geq 0,\hbox{ for every }x\in V.$$
Together with (\ref{support}) we then conclude that $J(u_1)>0$. Thus
\begin{equation}\label{rho2}
\rho_2:=\sup_{u\in\Phi^{-1}(]0,\infty[)}\frac{J(u)}{\Phi(u)}>0.
\end{equation}
Relations (\ref{rho1}) and (\ref{rho2}) finally imply that assertion (v) of Theorem \ref{Ricceri3cp} is also fulfilled. Put $\Lambda:=\frac1{\rho_2}$ (with the
convention $\frac1\infty:=0$). Note that if $g\colon V\times\R\to\R$ is continuous, then the map $\Psi\colon X\to \R$, defined by
$$ \Psi(u)=\int_VG(x,u(x))d\mu,$$
is, by the assertions a) and b) of Proposition \ref{compderivative}, a $C^1$-functional with compact derivative.
So, applying Theorem \ref{Ricceri3cp}
 and Proposition \ref{criticalpoint}, we  obtain the asserted conclusion.
\end{proof}

\begin{example}
Let $0<\alpha<2<\beta$ and define $f_1\colon\R\to\R$ by
$$f_1(t)=
\left\{
\begin{array}{l}
|t|^{\beta-2}t,\hbox{ if } |t|\leq1\\
|t|^{\alpha-2}t,\hbox{ if } |t|>1.
\end{array}
\right.
$$
Then $F_1\colon\R\to\R$, $F_1(t)=\int_0^tf_1(\xi)d\xi$, is given by
$$F_1(t)=
\left\{
\begin{array}{l}
\frac1\beta|t|^{\beta},\hbox{ if } |t|\leq1\\
\frac1\beta-\frac1\alpha+\frac1\alpha|t|^{\alpha},\hbox{ if } |t|>1.
\end{array}
\right.
$$
Consider a continuous map $a\colon V\to\R$ with $a(x)>0$, for every $x\in V$, and define $f\colon V\times\R\to\R$ by $f(x,t)=a(x)f_1(t)$. Then
$F\colon V\times\R\to\R$, $F(x,t)=\int_0^t f(x,\xi)d\xi$, is given by $F(x,t)=a(x)F_1(t)$. Hence $F$ satisfies condition $(C2)$ of Theorem \ref{main1}.
\end{example}

\section{A perturbed two-parameters Dirichlet problem}

Now we study, in a particular case, a perturbed version of the two-parameters problem $(DP_{\lambda,\eta})$ of the previous section.
More exactly, for fixed reals $r,s,q$ with $1<r<s<2<q$ and for the parameters $\lambda,\eta\geq0$, consider the following Dirichlet problem on the SG:
$$(P_{\lambda,\eta})
\left\{
\begin{array}{l}
-\Delta u(x)=\lambda |u(x)|^{s-2}u(x)-\eta |u(x)|^{r-2}u(x)+|u(x)|^{q-2}u(x),\ \forall x\in V\setminus V_0,\\
\\
u|_{V_0}=0,
\end{array}
\right.
$$
where we put, by definition, $|0|^{\ell}\cdot0:=0$, for every $\ell<0$.
By Proposition \ref{criticalpoint} and Remark \ref{panda}, the map $I_{\lambda,\eta}\colon H_0^1(V)\to\R$, defined by
\begin{equation}\label{ef}
I_{\lambda,\eta}(u)=\frac12||u||^2-\frac{\lambda}s\int_V|u|^sd\mu+\frac\eta r\int_V|u|^rd\mu-\frac1q\int_V|u|^qd\mu,
\end{equation}
is the energy functional attached to problem $(P_{\lambda,\eta})$. The derivative of this map is given, for every $u,v\in H_0^1(V)$, by
\begin{equation}\label{efderivative}
I'_{\lambda,\eta}(u)(v)={\mathcal W}(u,v)-\lambda\int_V|u|^{s-2}uvd\mu+\eta\int_V|u|^{r-2}uvd\mu-\int_V|u|^{q-2}uvd\mu.
\end{equation}

For the sake of completeness we recall the two mountain pass theorems that will be used to prove the main result of this section. The first one
is the celebrated mountain pass theorem due to Ambrosetti and Rabinowitz (e.g., Theorem 2.2 in \cite{Ra}):

\begin{theorem}\label{MPAR}
Let $X$ be a real Banach space and let $I\colon X\to\R$ be a $C^1$-functional satisfying the Palais-Smale condition. Furthermore assume that $I(0)=0$ and that
the following conditions hold:
\begin{itemize}
\item[{\rm (i)}] There are reals $\rho,\alpha>0$ such that $I|_{\partial B_\rho}\geq\alpha$.
\item[{\rm (ii)}] There is an element $e\in E\setminus B_\rho$ such that $I(e)\leq0$.
\end{itemize}
Then the real number $\kappa$, characterized as
\begin{equation}\label{criticalvalue}
\kappa:=\inf_{g\in\Gamma}\max_{t\in[0,1]}I(g(t)),
\end{equation}
where
$$\Gamma:=\{g\colon[0,1]\to X\mid\ g \hbox{ continuous},\ g(0)=0,\ g(1)=e\},$$ is a critical value of $I$ with $\kappa\geq\alpha$.
\end{theorem}

The next result generalizes the above Theorem by weakening condition (i). It goes back to P. Pucci and J. Serrin, and can be found
in \cite{PS}.

\begin{theorem}\label{MPPS}
Let $X$ be a real Banach space and let $I\colon X\to\R$ be a $C^1$-functional satisfying the Palais-Smale condition. Furthermore assume that $I(0)=0$ and that
the following conditions hold:
\begin{itemize}
\item[{\rm (i)}] There exists a real number  $\rho>0$   such that $I|_{\partial B_\rho}\geq 0$.
\item[{\rm (ii)}] There is an element $e\in E\setminus\overline B_\rho$ with $I(e)\leq 0$.
\end{itemize}
Then the real number $\kappa$ defined in {\rm (\ref{criticalvalue})}
is a critical value of $I$ with $\kappa\geq0$. If $\kappa=0$, there exists a critical point of $I$ on $\partial B_\rho$ corresponding to the critical value $0$.
\end{theorem}

We also recall two standard results concerning the existence of minimum points of sequentially weakly lower semicontinuous functionals.

\begin{proposition}\label{WI}
Let $X$ be a reflexive real Banach space, $M$ a bounded and sequentially weakly closed subset of $X$, and $f\colon M\to\R$ a sequentially weakly
lower semicontinuous functional. Then $f$ possesses at least one minimum point.
\end{proposition}

\begin{proposition}\label{WII}
Let $X$ be a reflexive real Banach space, $M$ a sequentially weakly closed subset of $X$, and $f\colon M\to\R$ a sequentially weakly
lower semicontinuous and coercive functional. Then $f$ possesses at least one minimum point.
\end{proposition}

We next establish some important properties of the energy functional $I_{\lambda,\eta}\colon H_0^1(V)\to\R$ attached to problem $(P_{\lambda,\eta})$.

\begin{lemma}\label{alsoforPS}
Let $\lambda,\eta\geq0$. If $(u_n)$ is a sequence in $H_0^1(V)$ such that both of the sequences $(I_{\lambda,\eta}(u_n))$ and $(I'_{\lambda,\eta}(u_n))$
are bounded, then $(u_n)$ is bounded, too.
\end{lemma}
\begin{proof}
Let $d$ be a real number such that $I_{\lambda,\eta}(u_n)\leq d$ and $||I'_{\lambda,\eta}(u_n)||\leq d$ for every index $n$. Relations (\ref{ef})
and (\ref{efderivative}) yield for every index $n$
\begin{eqnarray*}
&&I_{\lambda,\eta}(u_n)=\frac12||u_n||^2-\frac{\lambda}s\int_V|u_n|^sd\mu+\frac\eta r\int_V|u_n|^rd\mu-\frac1q\int_V|u_n|^qd\mu\\
&=&\left(\frac12-\frac1q\right)||u_n||^2-\lambda\left(\frac1s-\frac1q\right)\int_V|u_n|^sd\mu+\eta\left(\frac1r-\frac1q\right)\int_V|u_n|^rd\mu+
\frac1qI'_{\lambda,\eta}(u_n)(u_n).
\end{eqnarray*}
Using (\ref{embeddingconstant}), we get
$$d\geq I_{\lambda,\eta}(u_n)\geq \left(\frac12-\frac1q\right)||u_n||^2-\lambda\left(\frac1s-\frac1q\right)c^s||u_n||^s-\frac{d}q||u_n||.$$
Since
$$\lim_{t\to\infty}\left(\left(\frac12-\frac1q\right)t^2-\lambda\left(\frac1s-\frac1q\right)c^st^s-\frac{d}qt\right)=\infty,$$
we conclude that the sequence $(u_n)$ has to be bounded.
\end{proof}

\begin{proposition}\label{efproperties}
Let $\lambda,\eta\geq0$. The energy functional $I_{\lambda,\eta}\colon H_0^1(V)\to\R$ attached to problem $(P_{\lambda,\eta})$ has the following properties:
\begin{itemize}
\item[{\rm a)}] $I_{\lambda,\eta}$ is a $C^1$-functional.
\item[{\rm b)}] $u\in H_0^1(V)$ is a solution of problem $(P_{\lambda,\eta})$
if and only if $u$ is a critical point of $I_{\lambda,\eta}$.
\item[{\rm c)}] $I_{\lambda,\eta}$ is sequentially weakly lower semicontinuous.
\item[{\rm d)}] $I_{\lambda,\eta}$ satisfies the Palais-Smale condition.
\item[{\rm e)}] $0$ is a local minimum of $I_{\lambda,\eta}$.

\end{itemize}
\end{proposition}
\begin{proof}
The assertions a) and b) follow from Proposition \ref{criticalpoint}, while c) is a consequence of
Corollary \ref{Iwslsc}.

d) Consider a sequence
$(u_n)$ in $H_0^1(V)$ such that $(I_{\lambda,\eta}(u_n))$ is bounded
and such that $(I'(u_n))$ converges to $0$. By Lemma \ref{alsoforPS}
we know that $(u_n)$ is bounded, thus, in view of Corollary
\ref{forPS}, the sequence $(u_n)$ contains a convergent subsequence.
Hence $I_{\lambda,\eta}$ satisfies the Palais-Smale condition.

e) We know from (\ref{ef}) that for every $u\in H_0^1(V)$
$$I_{\lambda,\eta}(u)=\frac12||u||^2+\int_V\left(\frac\eta r-\frac{\lambda}s|u|^{s-r}-\frac1q|u|^{q-r}\right)|u|^rd\mu.$$
Let $h\colon\R\to\R$ be defined by $h(t)=\frac\eta r-\frac\lambda s|t|^{s-r}-\frac1q|t|^{q-r}$. Since $h(0)=\frac\eta r>0$ and since $h$ is continuous,
there exists $\delta>0$ such that $h(t)>0$ for every $t\in]-\delta,\delta[$. Put $r:=\frac\delta c$. If $u\in B_r$, then (\ref{embeddingconstant})
implies $||u||_{sup}\leq c||u||<\delta$, hence $u(x)\in]-\delta,\delta[$ for every $x\in V$. It follows that
$I_{\lambda,\eta}\geq0=I_{\lambda,\eta}(0),\hbox{ for every } u\in B_r,$
thus $0$ is a local minimum of $I_{\lambda,\eta}$.
\end{proof}

\begin{lemma}\label{ulambda}
If $\lambda>0$, then there exists a nonzero element $u_\lambda\in H_0^1(V)$ satisfying the following equality
\begin{equation}\label{ulambdaformula}
||u_\lambda||^2=\lambda\int_V|u_\lambda|^sd\mu.
\end{equation}
In particular, the inequality
\begin{equation}\label{ulambdainequality}
||u_\lambda||\leq(\lambda c^s)^{\frac1{2-s}}
\end{equation}
holds.
\end{lemma}
\begin{proof}
Let $\psi_\lambda\colon H_0^1(V)\to\R$ be defined by $\psi_\lambda(u)=\frac12||u||^2-\frac\lambda s\int_V|u|^sd\mu$. By (\ref{embeddingconstant}),
the following relations hold for every $u\in H_0^1(V)$
$$\psi_\lambda(u)\geq \frac12||u||^2-\frac\lambda s c^s||u||^s=||u||^2\left(\frac12-\frac\lambda sc^s||u||^{s-2}\right).$$
Since $s<2$, we have that $\displaystyle{\lim_{||u||\to\infty}}\psi_\lambda(u)=\infty$, i.e., $\psi_\lambda$ is coercive. From Corollary \ref{Iwslsc}
we know that $\psi_\lambda$ is
sequentially weakly lower semicontinuous, thus, by Proposition \ref{WII}, $\psi_\lambda$ admits at least one global minimum which we denote by $u_\lambda$.
Since $u_\lambda$ is a critical point of $\psi_\lambda$, we have that $\psi_\lambda'(u_\lambda)(v)=0$, for every $v\in H_0^1(V)$, i.e., in view
of Proposition \ref{criticalpoint},
$${\mathcal W}(u_\lambda,v)-\lambda\int_V|u_\lambda|^{s-2}u_\lambda vd\mu=0, \hbox{ for every } v\in H_0^1(V).$$
If we take $v=u_\lambda$ in the above equality, we get (\ref{ulambdaformula}). Inequality (\ref{embeddingconstant}) then yields
(\ref{ulambdainequality}).

We finally show that $u_\lambda$ is nonzero. For this fix an arbitrary nonzero element $v\in H_0^1(V)$. According to (\ref{support}), we have that
$\int_V|v|^sd\mu>0$. On the other hand, the following equality holds for every $t>0$
$$\psi_t(tv)=t^s\left(\frac12t^{2-s}||v||^2-\frac\lambda s\int_V|v|^sd\mu\right).$$
For $t$ sufficiently small (more exactly, for $0<t^{2-s}<\frac{2\lambda}{s||v||^2}\int_V|v|^sd\mu$) we thus have that
$$\psi_\lambda(u_\lambda)\leq \psi_\lambda(tv)<0=\psi_\lambda(0),$$
hence $u_\lambda$ is nonzero.
\end{proof}
Now we can state the main result of this section.

\begin{theorem}\label{main2}
There exists a real number $\Lambda>0$ with the following property: For every $\lambda\in]0,\Lambda[$ there exists $\eta_\lambda>0$ such that
for each $\eta\in[0, \eta_\lambda[$ the problem $(P_{\lambda,\eta})$ has at least three nonzero solutions.
\end{theorem}
\begin{proof}
Let $R:=c^{-\frac q{q-2}}$, so $c^qR^q=R^2$. Put $m:=\frac12(\frac12-\frac1q)R^2$. Obviously $m>0$.
For every $\lambda,\eta\geq0$ and every $u\in H_0^1(V)$ with $||u||=R$ we then have, according to (\ref{embeddingconstant}) and (\ref{ef}),
\begin{equation}\label{choiceLambda}
I_{\lambda,\eta}(u)\geq \frac12R^2-\frac\lambda sc^sR^s-\frac1qc^qR^q=\left(\frac12-\frac1q\right)R^2-\frac\lambda sc^sR^s=2m-\frac\lambda sc^sR^s.
\end{equation}
Consider
$$\Lambda:=\min\left\{\frac{ms}{c^sR^s},\frac{R^{2-s}}{c^s}\right\}.$$
Fix now an arbitrary $\lambda\in]0,\Lambda[$. From (\ref{choiceLambda}) we then get that
\begin{equation}\label{tromp}
\inf_{||u||=R}I_{\lambda,\eta}(u)>m,\hbox{ for every }\eta\geq0.
\end{equation}
Also, by (\ref{ulambdainequality}) and the choice of $\Lambda$, we have that
\begin{equation}\label{snorry}
||u_\lambda||<R.
\end{equation}
By (\ref{ef}) and (\ref{ulambdaformula}), the following equality holds for every $\eta\geq0$
$$I_{\lambda,\eta}(u_\lambda)=||u_\lambda||^2\left(\frac12-\frac1s\right)+\frac\eta r\int_V|u_\lambda|^rd\mu-\frac1q\int_v|u_\lambda|^qd\mu.$$
Since $u_\lambda$ is nonzero, (\ref{support}) implies that $\int_V|u_\lambda|^rd\mu>0$. Put
$$\eta_{1,\lambda}:=\frac{r\left(||u_\lambda||^2\left(\frac1s-\frac12\right)+\frac1q\int_v|u_\lambda|^qd\mu\right)}{\int_V|u_\lambda|^rd\mu}.$$
Then $\eta_{1,\lambda}>0$ and
\begin{equation}\label{falf}
I_{\lambda,\eta}(u_\lambda)<0,\hbox{ for every }\eta\in[0,\eta_{1,\lambda}[.
\end{equation}
For $\eta\geq 0$ we have, by (\ref{ef}) and (\ref{ulambdaformula}),
\begin{equation}\label{forlimes}
I_{\lambda,\eta}(tu_\lambda)=\frac12t^2||u_\lambda||^2-\frac{t^s}s||u_\lambda||^2+\frac{\eta t^r}r\int_V|u_\lambda|^rd\mu-\frac{t^q}q\int_V|u_\lambda|^qd\mu,
\hbox{ for }t\geq0.
\end{equation}
Since $s<2$ we have $\frac{t^2}2\leq\frac{t^s}s$, for every $t\in[0,1]$. Thus
$$I_{\lambda,\eta}(tu_\lambda)\leq \frac{\eta t^r}r\int_V|u_\lambda|^rd\mu,\hbox { for }\eta\geq0,\hbox{ and }t\in[0,1].$$
Let
$$\eta_{2,\lambda}:=\frac{mr}{\int_V|u_\lambda|^rd\mu}.$$
Then $\eta_{2,\lambda}>0$ and
\begin{equation}\label{pleosc}
I_{\lambda,\eta}(tu_\lambda)<m,\hbox{ for every }\eta\in[0,\eta_{1,\lambda}[\hbox{ and every }t\in[0,1].
\end{equation}
Define $\eta_\lambda:=\min\{\eta_{1,\lambda},\eta_{2,\lambda}\}$ and pick an arbitrary $\eta\in[0,\eta_\lambda[$. Relation (\ref{falf}) implies
\begin{equation}\label{plitsch}
I_{\lambda,\eta}(u_\lambda)<0,
\end{equation}
while relation (\ref{pleosc}) yields
\begin{equation}\label{koala}
I_{\lambda,\eta}(tu_\lambda)<m,\hbox{ for every }t\in[0,1].
\end{equation}
We next proceed in three steps to get three nonzero solutions of problem $(P_{\lambda,\eta})$. From the assertions a) and d) of Proposition
\ref{efproperties} we know that
$I_{\lambda,\eta}$ is a $C^1$-functional which satisfies the Palais-Smale condition.

{\it The first step}: The closed ball $\overline B_R$ is weakly closed (being convex and closed in the strong topology), hence it is also sequentially
weakly closed. By assertion c) of Proposition \ref{efproperties} the restriction $I_{\lambda,\eta}|_{\overline B_R}$ is sequentially weakly lower semicontinuous.
Proposition \ref{WI} implies then that $I_{\lambda,\eta}|_{\overline B_R}$ has at least one minimum point $\widetilde u_1$. By (\ref{snorry}) we know that
$u_\lambda\in B_R$, thus $I_{\lambda,\eta}(\widetilde u_1)\leq I_{\lambda,\eta}(u_\lambda)$. Since $I_{\lambda,\eta}(u_\lambda)<0=I_{\lambda,\eta}(0)$
(by inequality
\ref{plitsch}), we conclude that $\widetilde u_1$ is nonzero. Also, in view of (\ref{tromp}), we have that $\widetilde u_1\in B_R$. This shows that
$\widetilde u_1$ is a local minimum point, hence a critical point, of $I_{\lambda,\eta}$.

{\it The second step}: From assertion e) of Proposition \ref{efproperties} we get a positive real $r<||u_\lambda||$ such that
$0=I_{\lambda,\eta}(0)\leq I_{\lambda,\eta}(u)$, for every $u\in \overline B_r$. From (\ref{plitsch}) we know that  $I_{\lambda,\eta}(u_\lambda)<0$, so
Theorem \ref{MPPS} guarantees the existence of a nonzero critical
point $\widetilde u_2$ of $I_{\lambda,\eta}$ such that $I_{\lambda,\eta}(\widetilde u_2)\geq0$ and
$$I_{\lambda,\eta}(\widetilde u_2)=\inf_{g\in\Gamma}\max_{t\in[0,1]}I_{\lambda,\eta}(g(t)),$$
where
$$\Gamma:=\{g\colon[0,1]\to H_0^1(V)\mid\ g \hbox{ continuous},\ g(0)=0,\ g(1)=u_\lambda\}.$$
Choosing $\widetilde g\colon[0,1]\to H_0^1(V)$, $\widetilde g(t)=tu_\lambda$, we get
in virtue of (\ref{koala})
$$I_{\lambda,\eta}(\widetilde u_2)\leq \max_{t\in[0,1]}I_{\lambda,\eta}(tu_\lambda)<m.$$

{\it The third step}: Since $\int_V|u_\lambda|^qd\mu>0$, relation (\ref{forlimes}) yields $\displaystyle{\lim_{t\to\infty}I_{\lambda,\eta}(tu_\lambda)=-\infty}$.
Thus there is a positive real $t$ such that $I_{\lambda,\eta}(tu_\lambda)<0$ and $t||u_\lambda||>R$. Taking into account (\ref{tromp}), Theorem \ref{MPAR}
implies the existence of a critical point $\widetilde u_3$ of $I_{\lambda,\eta}$ such that $I_{\lambda,\eta}(\widetilde u_3)\geq m$.

Thus $\widetilde u_1, \widetilde u_2,$ and $\widetilde u_3$ are nonzero critical points of $I_{\lambda,\eta}$ satisfying the inequalities
$$I_{\lambda,\eta}(\widetilde u_1)<0\leq I_{\lambda,\eta}(\widetilde u_2)<m\leq I_{\lambda,\eta}(\widetilde u_3).$$
Hence $\widetilde u_1, \widetilde u_2, \widetilde u_3$ are pairwise distinct. So assertion b) of Proposition
\ref{efproperties} finally implies that problem $(P_{\lambda,\eta})$ has at least three nonzero solutions.
\end{proof}

\begin{remark}
As already mentioned in the introduction, in \cite{An} there is investigated the analogous of problem $(P_{\lambda,\eta})$ in case of the $p$-Laplacian.
Since $H_0^1(V)$ satisfies (\ref{embedding}), our situation differs from that one in \cite{An}, where the corresponding result to Theorem \ref{main2}
has been obtained only for subcritical values of $q$, while for the critical value or for supercritical values of $q$ one can guarantee only the
existence of at least two nonzero solutions (see Theorem 1, respectively, Theorem 2 in \cite{An}). Note that our Theorem \ref{main2} holds for every $q>2$.
\end{remark}

\subsection*{Acknowledgements}
Varga was fully supported by the grant CNCSIS PCCE-55/2008
``Sisteme diferen\c{t}iale \^\i n analiza neliniar\u a \c si aplica\c tii. Breckner
was supported by the grant CNMP PN-II-P4-11-020/2007. Repov\v s was supported by SRA grants P1-0292-0101, J1-2057-0101 and J1-4144-0101.

\end{document}